\documentclass[11pt]{amsart}
\usepackage{amsfonts,amsmath,amssymb,amsthm,color,mathtools}
\usepackage{graphicx} 
\usepackage[margin=1in]{geometry}
\usepackage[style=numeric]{biblatex} 
\theoremstyle{definition}
\newtheorem{thm}{Theorem}
\newtheorem{lem}[thm]{Lemma}
\newtheorem{prop}[thm]{Proposition}
\newtheorem{cor}[thm]{Corollary}
\newtheorem{example}[thm]{Example}
\newtheorem{remark}[thm]{Remark}
\newtheorem{definition}[thm]{Definition}
\newtheorem{question}[thm]{Question}
\numberwithin{thm}{section}

\newtheorem*{QsymTheorem}{Theorem \ref{thm:qsym-sym}}
\newtheorem*{DagTheorem}{Theorem \ref{thm:dag}}
\newtheorem*{MountainTheorem}{Theorem \ref{thm:mountains}}

\newcommand{\Sym}{\mathrm{Sym}}
\newcommand{\QSym}{\mathrm{QSym}}
\newcommand{\N}{\mathbb{N}}
\newcommand{\rev}{\mathrm{rev}}
\newcommand{\asc}{\mathrm{asc}}
\newcommand{\cycle}{\mathrm{cycle}}
\newcommand{\reflect}{\mathrm{reflect}}
\newcommand{\swap}{\mathrm{swap}}
\newcommand{\stat}{\mathrm{stat}}

\title{When is the Chromatic Quasisymmetric Function Symmetric?}
\author{Maria Gillespie} 
\address{Department of Mathematics, Colorado State University, Fort Collins, CO 80523, USA}
\email{maria.gillespie@colostate.edu}
\thanks{All authors were partially supported by NSF DMS award number 2054391.}

\author{Joseph Pappe}
\address{Department of Mathematics, Colorado State University, Fort Collins, CO 80523, USA}
\email{joseph.pappe@colostate.edu}

\author{Kyle Salois}
\address{Department of Mathematics, Colorado State University, Fort Collins, CO 80523, USA}
\email{kyle.salois@colostate.edu}
\date{\today}

\date{\today}

\begin{document}

\begin{abstract}
    We investigate the problem of when a chromatic quasisymmetric function (CQF) $X_G(x;q)$ of a graph $G$ is in fact symmetric. We first prove the remarkable fact that if a product of two quasisymmetric functions $f$ and $g$ in countably infinitely many variables is symmetric, then in fact $f$ and $g$ must be symmetric.  This allows the problem to be reduced to the case of connected graphs.

    We then show that any labeled graph having more than one source or sink has a nonsymmetric CQF.  As a corollary, we find that all trees other than a directed path have a nonsymmetric CQF. We also show that a family of graphs we call ``mixed mountain graphs'' always have symmetric CQF.   
\end{abstract}

\maketitle

\section{Introduction}
Chromatic symmetric functions of graphs are a topic of much recent study, especially in light of the Stanley--Stembridge and Shareshian--Wachs Conjectures.  The \textbf{chromatic symmetric function} of a graph $G$ on $n$ vertices $1,2,\ldots,n$ is defined to be $$X_G(x)=X_G(x_1,x_2,\ldots)=\sum_{\substack{\kappa:[n]\to \mathbb{N} \\ \text{ proper}}} x_{\kappa(1)}x_{\kappa(2)}\cdots x_{\kappa(n)}$$ where $\kappa$ is a proper coloring of $G$ with colors from $\mathbb{N}$.   The \textbf{chromatic quasisymmetric function (CQF)} $X_G(x;q)$ is a $q$-analog of $X_G(x)$ defined as $$X_G(x;q)=\sum_\kappa q^{\asc(\kappa)}x_{\kappa(1)}\cdots x_{\kappa(n)}$$ where $\asc(\kappa)$ is the number of adjacent pairs $(i,j)$ of vertices with $i<j$ and $\kappa(i)<\kappa(j)$ (such pairs are called \textit{ascents} of $\kappa$).

\begin{example}\label{ex:1}
    The CQF of the path graph of length $1$ shown in Figure \ref{fig:P1} is $$(1+q)(x_1x_2+x_1x_3+x_2x_3+\cdots)=(1+q)e_2$$ where $e_2$ is the elementary symmetric function of degree $2$.
\end{example}
\begin{figure}[h]
    \begin{center}
\includegraphics{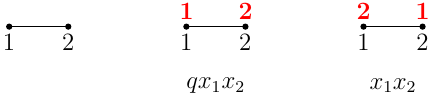}
    \end{center}\caption{
    The path graph of length $1$ with its labels is shown at left above; two colorings are shown at middle and right, with the middle having one ascent and the right hand diagram having none.}
    \label{fig:P1}
\end{figure}

The Stanley--Stembridge conjecture \cite{Stanley95,StanleyStembridge} states that certain chromatic symmetric functions are \textit{$e$-positive}, meaning that they expand positively in the basis of elementary symmetric functions (see Section \ref{sec:symmetric-functions-background}).  In particular, the subset of graphs in this conjecture are the \textit{incomparability graphs} of posets that are ``3+1--avoiding'', meaning they have no induced subposet isomorphic to the disjoint union of a chain of length $3$ and a singleton. Gasharov \cite{Gasharov} showed that for such graphs, $X_G(x)$ expands positively in the Schur basis, and a large body of modern work has shown that special sub-families of such graphs exhibit the desired $e$-positivity \cite{AN21,AWvW21,CH22,CH19,CMP23,Dahlberg18,DvW18,FHM19,GS01,HP19,HNY20,LY21,MPW24,Stanley95,Tom23,WW23,Wang22}. Recently, a proof of the Stanley--Stembridge conjecture was proposed by Hikita in \cite{Hik24}, and another soon followed by Griffin et al.\ in \cite{Griffin}.

The Shareshian--Wachs conjecture \cite{SW16} refines this conjecture in a special case.  An important example of graphs coming from $3+1$-free posets are \textbf{unit interval graphs}, whose vertex set is a collection of unit length intervals on the real line, and whose edges correspond to overlapping intervals.  For such graphs, Shareshian and Wachs conjectured that $X_G(x;q)$ is also $e$-positive, in the sense that its coefficients in the $e$ basis are in $\mathbb{Z}_+[q]$.  They also conjectured a connection to the cohomology rings of \textit{Hessenberg varieties}, which was later proven by Brosnan and Chow \cite{BrosnanChow} and independently by Guay-Paquet \cite{GP16}. Moreover, a result by Guay-Paquet \cite{GP13} shows that the Shareshian--Wachs conjecture on unit interval graphs implies the Stanley--Stembridge conjecture for all $3+1$-free posets.

Both the Stanley--Stembridge conjecture and the Shareshian--Wachs conjecture have been proven in a number of notable special cases, involving various infinite subfamilies of graphs \cite{AN21,CH19,CMP23,HP19,HNY20,SW16,Tom23}.  There are also additional families of graphs whose CQF is known to be $e$-positive, most notably the cycle graphs $C_n$ consisting of $n$ vertices connected in a length $n$ cycle \cite{EW20}.   However, not all CQF's are $e$-positive, and in particular they are not even necessarily symmetric functions, as shown in the following example.

\begin{example}
   Consider the graph $G$ on $\{1,2,3\}$ shown in Figure \ref{fig:P2}.  The coefficient of the maximal power of $q$ in $X_G(x;q)$, namely $q^2$, is $$2e_3+(x_1^2x_2+x_1^2x_3+x_2^2x_3+\cdots)=2e_3+M_{2,1}$$ where $M_{2,1}$ is the monomial quasisymmetric function, which is not symmetric.
\end{example}

\begin{figure}
    \centering
    \includegraphics{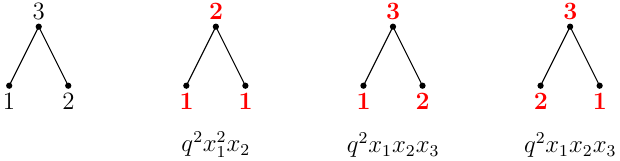}
    \caption{A graph $G$ shown at left, for which $X_G(x;q)$ is not symmetric.  Three of its colorings with maximal number of ascents are shown to the right.}
    \label{fig:P2}
\end{figure}

Due to what is known about the cases of cycle graphs and unit interval graphs, we pose the following question.

\begin{question}
    Is every CQF that is symmetric also $e$-positive?
\end{question}

In order to begin to address this question, here we explore the problem of determining when a CQF is symmetric.  We first reduce the problem to the case of connected graphs using our first main result:

\begin{thm}\label{thm:qsym-sym}
    Suppose $f$ and $g$ are quasisymmetric functions in countably infinitely many variables and $f\cdot g$ is symmetric.  Then both $f$ and $g$ are symmetric.
\end{thm}

Note that this does not hold for finitely many variables; we have that $x_1^2x_2\cdot x_1x_2^2=x_1^3x_2^3$ is symmetric in two variables.   Since CQF's are multiplicative across disjoint union of graphs, we immediately have the following.

\begin{cor}
    Suppose $G$ is a graph with two or more connected components, and $X_G(x;q)$ is symmetric.  Then the CQF of each connected component is symmetric.
\end{cor}

Some prior progress on the classification of symmetric CQF's includes the observation in \cite{SW16} that if a CQF is symmetric, then its coefficients in the basis of monomial quasisymmetric functions $M_\alpha$ are \textit{palindromic} polynomials in $q$ (when reading off the coefficients of each $q^i$).  More recently, in the preprint \cite{Steph}, Aliniaeifard, Asgarli, Esipova, Shelburne, van Willigenburg, and Whitehead McGinley showed that for any non-standard orientation of the path graph and for any orientation of the star graph, the resulting CQF is not symmetric. Our second main result generalizes this to all directed acyclic graphs with more than one sink or source.

\begin{thm}\label{thm:dag}
    The CQF of any connected, directed acyclic graph with more than one sink or source is not symmetric.
\end{thm}

As a corollary, we obtain an answer to an open question posed by Aliniaeifard, Asgarli, Esipova, Shelburne, van Willigenburg, and Whitehead McGinley \cite{Steph} of which oriented trees have a symmetric CQF.

\begin{cor}
    The CQF of a tree is symmetric if and only if the tree is a directed path.
\end{cor}

Our final main result is as follows.  Define a \textbf{$k$-mountain} to be a complete $k$-clique, and a \textbf{bottomless $k+1$-mountain} to be a $k+1$-clique with one edge removed.  To make a \textbf{mixed mountain graph}, we string together any sequence of $k$-mountains and bottomless $k+1$-mountains, where each pair of adjacent mountains shares a single vertex, and if it is a bottomless mountain the shared vertex is one of the vertices of the missing edge.  Then we connect the final endpoints of the first and last mountains by a single extra edge, as shown in Figure \ref{fig:mixed-mountain}.  We orient all edges from left to right, corresponding to the left to right labeling of the vertices as shown with $1,2,3,\ldots,n$.

\begin{figure}
    \centering
    \includegraphics{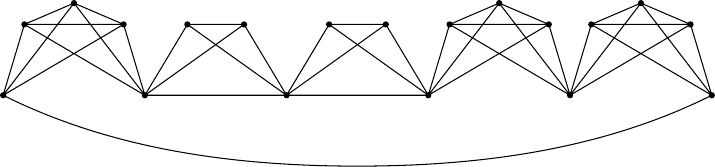}
    \caption{A mixed mountain graph for $k=4$.}
    \label{fig:mixed-mountain}
\end{figure}

\begin{thm}\label{thm:mountains}
  The CQF of every mixed mountain graph is symmetric.
\end{thm}

We have used Sage \cite{sage} to test all labeled, connected graphs up to $8$ vertices, and have found that the mixed mountain graphs and unit interval orders describe all symmetric CQF's for these numbers. The mixed mountain graphs whose CQF we were able to compute in Sage also all had $e$-positive CQF's.  We generally ask whether there are more families of graphs for larger vertices that have a symmetric CQF.

\subsection{Outline}

In Section \ref{sec:background} we establish background and notation on quasisymmetric functions and CQF's. We then prove Theorem \ref{thm:qsym-sym} in Section \ref{sec:qsym-products}.  In Section \ref{sec:dag}, we prove Theorem \ref{thm:dag} as well as provide a number of general results on when a CQF is nonsymmetric in general. We prove Theorem \ref{thm:mountains} in Section \ref{sec:mountains}, starting with the case of ordinary $k$-mountain graphs with no bottomless $k+1$-cliques, and then generalizing to the full case.

\subsection{Acknowledgements}

We thank Eugene Gorsky, Kevin Liu, and Richard Stanley for helpful conversations pertaining to this work.

\section{Notation and Background}\label{sec:background}

We briefly establish some notation on the combinatorial objects at play.

\subsection{Quasisymmetric and symmetric functions}\label{sec:symmetric-functions-background}

We first recall quasisymmetric functions, in both finite and infinitely many variables.

\begin{definition}
A \textbf{quasisymmetric function} (over the coefficient ring $R$) in $x_1,\ldots,x_n$ is a polynomial $f(x_1,\ldots,x_n)\in R[x_1,\ldots,x_n]$ such that for any composition $\alpha=(\alpha_1,\ldots,\alpha_n)$ with at most $n$ parts, the coefficient of $x_1^{\alpha_1}\cdots x_m^{\alpha_m}$ is the same as that of $x_{i_1}^{\alpha_1}\cdots x_{i_m}^{\alpha_m}$ for any $$1\le i_1<i_2<\cdots <i_m\le n.$$ 

We write $\QSym_R[x_1,\ldots,x_n]$ for the ring of quasisymmetric functions in $n$ variables.
\end{definition}

\begin{definition}
    We write $$\QSym_R=\QSym_R[x_1,x_2,\ldots]$$ for the inverse limit, under the maps $\QSym_R[x_1,\ldots,x_n]\to \QSym_R[x_1,\ldots,x_{n-1}]$ formed by setting $x_n=0$, of the spaces $\QSym_R[x_1,\ldots,x_n]$.  
    
    This is called the ring of quasisymmetric functions over $R$, and can be thought of as the ring of bounded-degree sums of monomials such that the coefficient of $x_{i_1}^{\alpha_1}\cdots x_{i_m}^{\alpha_m}$ for any composition $\alpha$ and any increasing sequence $i_j$ is only dependent on $\alpha$.
\end{definition}

Throughout this paper, we will use $R=\mathbb{Z}$.  There are several natural bases for $\QSym_\mathbb{Z}$, one being the \textbf{monomial quasisymmetric functions} $$M_\alpha=\sum_{i_1<\cdots<i_m}x_{i_1}^{\alpha_1}\cdots x_{i_m}^{\alpha_m}.$$

\begin{definition}
    A function $f\in R[x_1,\ldots,x_n]$ is \textbf{symmetric} if $$f(x_{\pi(1)},\ldots,x_{\pi(n)})=f(x_1,\ldots,x_n)$$ for all permutations $\pi\in S_n$.  We write $\Sym_R[x_1,\ldots,x_n]$ for the ring of symmetric polynomials in $n$ variables over the coefficient ring $R$. 

    As for quasisymmetric functions, we write $\Sym_R=\Sym_R[x_1,x_2,\ldots]$ for the inverse limit over $n$ under the compatibility maps formed by setting $x_n=0$.
\end{definition}

A natural basis of $\Sym_{\mathbb{Z}}$ is the \textbf{monomial symmetric functions}
$$m_\lambda=x_1^{\lambda_1}x_2^{\lambda_2}\cdots x_n^{\lambda_n} +\text{similar terms},$$ where the similar terms consist of all monomials whose multiset of exponents are the parts of the partition $\lambda$.  Here, a partition $\lambda=(\lambda_1,\ldots,\lambda_n)$ differs from a composition $\alpha$, since we require that $\lambda_1\ge \cdots \ge \lambda_n$.

Another important basis of $\Sym_{\mathbb{Z}}$ is the following.

\begin{definition}
The \textbf{elementary symmetric function} $e_n$ is the monomial symmetric function $m_{(1,1,1,\ldots,1)}$ for the partition of length $n$ whose parts are all $1$.  For instance, $e_2=x_1x_2+x_1x_3+x_2x_3+\cdots.$  For a  partition $\lambda=(\lambda_1,\ldots,\lambda_k)$, the basis element $e_\lambda$ is defined to be $e_{\lambda_1}\cdots e_{\lambda_k}$.
\end{definition}

A symmetric function is \textbf{$e$-positive} if its expansion in the $e_\lambda$ basis has all nonnegative integer coefficients. A final important basis is the \textbf{Schur basis} $s_\lambda$; we will not need to define it here, but will occasionally mention the important problem of determining \textbf{Schur positivity}, which is always implied by $e$-positivity.

\subsection{Graph colorings and CQF's}

Throughout this paper we work with vertex-labeled graphs $G$, with vertices labeled $1,2,\ldots,n$, and we write $[n]=\{1,2,\ldots,n\}$.  The labeling automatically induces an acyclic orientation on the edges of $G$, where an edge points from the smaller vertex label to the larger.

\begin{definition}
    A \textbf{proper coloring} of $G$ is a map $\kappa:[n]\to \mathbb{N}$ where $\mathbb{N}=\{1,2,3,\ldots\}$, such that any two adjacent vertices labeled $u,v$ have $\kappa(u)\neq \kappa(v)$.
\end{definition}  

\begin{definition}
An \textbf{ascent} of a proper coloring is an edge from $u$ to $v$ such that $u<v$ and $\kappa(u)<\kappa(v)$.  We write $\asc(\kappa)$ for the number of ascents.    
\end{definition}

\begin{definition}
    The \textbf{weight} of a proper coloring is the tuple $(\kappa^{-1}(1),\kappa^{-1}(2),\kappa^{-1}(3),\ldots)$ where we often truncate to make a finite tuple after the maximum number $n$ used in the proper coloring.  That is, it is the tuple of multiplicities of each color appearing in the coloring.
\end{definition}

For example, a coloring with weight $(5,2,0,3)$ has five vertices colored $1$, two vertices colored $2$, and three vertices colored $4$.  We will often use exponential notation to indicate repeated entries in a weight; for instance, $(1^3,2,1^4)$ is a shorthand representation for the tuple $(1,1,1,2,1,1,1,1)$.

As mentioned in the introduction, we define the \textbf{chromatic quasisymmetric function}, or $\textbf{CQF}$, of $G$ to be $$X_G(x;q)=\sum_\kappa q^{\asc(\kappa)}x_{\kappa(1)}\cdots x_{\kappa(n)}$$ where the sum ranges over all proper colorings $\kappa$ of $G$. Note that the CQF matches the chromatic symmetric function at the specialization $q=1$.

We will often use the following fact.

\begin{remark}
   For a weight tuple $w$, the coefficient of the monomial quasisymmetric function $M_{w}$ is the sum of the monomials $q^{\asc(\kappa)}$ over all proper colorings $\kappa$ of weight $w$.
\end{remark}

\section{Symmetric products of quasisymmetric functions}\label{sec:qsym-products}

In this section, we show that if the product of two quasisymmetric functions $f,g$ is symmetric, then in fact both of $f,g$ are symmetric.

\begin{lem}\label{lem:extend}
    Let $R$ be any ring.  If $f\in R[x_1,\ldots,x_n]$ is irreducible, then it is also irreducible in $R[x_1,\ldots,x_n,x_{n+1},\ldots,x_m]$. 
\end{lem}

\begin{proof}
    By induction, it suffices to prove the statement for $m=n+1$.  
    
    Now, suppose for contradiction that $f=gh$ in $R[x_1,\ldots,x_n,x_{n+1}]$ where $g,h$ both have degree at least $1$.  Plug in $x_{n+1}=0$ on both sides of the equation.  Since $f$ only depends on $x_1,\ldots,x_n$, the left hand side is still simply $f$.  On the right hand side, if neither $g$ nor $h$ become a constant, then $f$ factors in $R[x_1,\ldots,x_n]$ and we have a contradiction. 
    
    On the other hand, if $g$ or $h$ become a constant at $x_{n+1}=0$, say $g|_{x_{n+1}=0}=c\in R$, then $g=c+x_{n+1}r$ for some nonzero $r\in R[x_1,\ldots,x_{n+1}]$.  Consider both $g$ and $h$ as polynomials in $x_{n+1}$ with coefficients in $R[x_1,\ldots,x_n]$, and consider the highest power of $x_{n+1}$ in each.  The product of these terms gives the highest power of $x_{n+1}$ in $f=gh$, which is nonvanishing, a contradiction since $f$ does not depend on $x_{n+1}$. 
\end{proof}

We now use a theorem of \cite{Hazewinkel} that establishes generators for $\QSym$ using $\lambda$-ring theory.  A \textbf{Lyndon word} is a word $\alpha=\alpha_1,\ldots,\alpha_m$ of positive integers where $\alpha$ is strictly lexicographically smaller than any cycling $\alpha_i,\ldots,\alpha_m,\alpha_1,\ldots,\alpha_{i-1}$.  In \cite{Hazewinkel}, Hazewinkle defines quasisymmetric functions $\lambda_n(M_\alpha)$ for each Lyndon word $\alpha$ and positive integer $n$, of degree $n\sum_i \alpha_i$, as follows:

$$\lambda_n(M_\alpha)=\frac{1}{n!}\det\begin{pmatrix}
   M_\alpha & 1 & 0 & 0 & \cdots & 0 \\
   M_{2\alpha} & M_{\alpha} & 2 & 0 &\cdots & 0 \\
   M_{3\alpha} & M_{2\alpha} & M_{\alpha} & 3 &  \cdots & 0 \\
   \vdots & & & & \ddots & \vdots \\
   M_{n\alpha} & M_{(n-1)\alpha} & M_{(n-2)\alpha} &  M_{(n-3)\alpha} &\cdots & n-1 \\   
\end{pmatrix}$$
where $k\alpha$ is the composition $(k\alpha_1, k\alpha_2, \ldots, k \alpha_m)$.

\begin{thm}[\cite{Hazewinkel}, Theorem 3.1] \label{thm:hazewinkle} We have $\QSym_{\mathbb{Z}}=\mathbb{Z}[\lambda_n(M_\alpha)]$ where $\alpha$ ranges over all Lyndon words with $\gcd(\alpha_i)=1$, and $\lambda_n(M_{1})=e_n$ for all $n$.
\end{thm}

\begin{cor}\label{cor:generators}
   We have $\QSym_{\mathbb{Z}}=\mathbb{Z}[e_1,e_2,\ldots,f_1,f_2,\ldots]$ for some functions $f_i$ such that there are finitely many $f$'s of each degree.
\end{cor}

\begin{proof}
    The generators $\lambda_n(M_\alpha)$ give the elementary symmetric functions for $\alpha=(1)$, and there are a finite number of generators $\lambda_n(M_\alpha)$ of a given degree $K$, since for it to have degree $K$ we must have $|\alpha|=K/n$.
\end{proof}

\begin{QsymTheorem}
    Suppose $f\cdot g=h$ where $f$ and $g$ are quasisymmetric and $h$ is symmetric.  Then $f$ and $g$ are in fact symmetric.
\end{QsymTheorem}

\begin{proof}
 We induct on the degree of $h$.  If $h$ has degree $0$, then $f$ and $g$ are also constants and are therefore symmetric.

 Now, suppose $h$ has degree $d$, and assume for strong induction that the statement is true for all degrees less than $d$.  We consider two cases.

 \textbf{Case 1.} Suppose $h$ is irreducible in $\Sym=\mathbb{Z}[e_1,e_2,\ldots]$.  Then it is irreducible over $\mathbb{Z}[e_1,\ldots,e_d]$ because $h$ has degree $d$.  Thus, by Lemma \ref{lem:extend}, $h$ is irreducible in $\mathbb{Z}[e_1,\ldots,e_d,f_1,\ldots,f_N]$ where $f_1,\ldots,f_N$ are the generators $\lambda_n(M_\alpha)$ other than the elementary symmetric functions that have degree $\le d$ (see Corollary \ref{cor:generators}).  It follows from Theorem \ref{thm:hazewinkle} that $h$ is irreducible in $\QSym$, and so the factorization does not occur.

 \textbf{Case 2.}  Suppose $h$ factors in $\Sym$ as $h=u\cdot v$ where $u,v\in \Sym$ both have lower degree than $d$.  Then by Theorem \ref{thm:hazewinkle}, we know $\QSym$ is a unique factorization domain (since there are a finite number of generators of each degree, so any given polynomial lies in a finitely generated ring that has unique factorization).  Thus we can consider the unique factorizations   $$u=a_1a_2\cdots a_r\hspace{2cm} v=b_1b_2\cdots b_s$$ in $\QSym$.  By the induction hypothesis, all of $a_1,\ldots,a_r$ and $b_1,\ldots,b_s$ are symmetric.   Therefore, any factorization of $h=a_1\cdots a_r b_1\cdots b_s$ into a product of two factors $f,g$ must have that $f,g$ are products of disjoint subsets of the $a$ and $b$ factors up to scaling by a unit.  Thus $f$ and $g$ are both symmetric as well, as desired.
\end{proof}

\begin{remark}
    Note that Theorem~\ref{thm:qsym-sym} does not hold when restricted to finitely many variables. For instance, the monomials $x_1^nx_2^{n-1}\ldots x_{n}^{1}$ and $x_1^{1}x_2^2\ldots x_{n}^{n}$ are both elements of $\QSym[x_1,\ldots,x_n]$ and not of $\Sym[x_1,\ldots,x_n]$. However, their product $(x_1x_2\ldots x_n)^{n+1}$ is an element of $\Sym[x_1,\ldots,x_n]$.
\end{remark}

\begin{cor}
   Suppose $G$ is a graph with multiple connected components $G_1,\ldots,G_k$.  Then its chromatic quasisymmetric function $X_G(x;q)$ is symmetric if and only if each $X_{G_i}(x;q)$ is symmetric.
\end{cor}

\begin{proof}
    We have that $X_G(x;q)$ is the product of the functions $X_{G_i}(x;q)$.  The result then follows from Theorem \ref{thm:qsym-sym}.
\end{proof}

We now use the following theorem from \cite{Lam} to obtain a more general result.  We follow their definition of the full ring $K$ of bounded-degree formal sums of monomials in infinitely many variables:

\begin{definition}
    We define $K=\mathbb{Z}[[x_1,x_2,\ldots]]$ to be the ring of formal power series of bounded degree in infinitely many variables $x_i$.
\end{definition}

Lam and Pylyavskyy showed that $K$ is a unique factorization domain, as well as showing the following theorem.

\begin{thm}[\cite{Lam}, Theorem 8.1]\label{thm:Lam}
Suppose $f \in \QSym$ and $f = \prod_i f_i$ is a factorization of $f$ into irreducibles in $K$.
Then $f_i \in \QSym$ for each $i$.
\end{thm}

We can therefore use this and our results above to conclude the following.

\begin{cor}\label{cor:factors}
  If $h$ is a symmetric function and factors as $h=f\cdot g$ where $f,g$ are both arbitrary power series in $K$, then $f$ and $g$ are both symmetric functions.    
\end{cor}

\begin{proof}
    By Theorem \ref{thm:Lam}, we find that $f$ and $g$ must be quasisysmmetric, and by Theorem \ref{thm:qsym-sym} it follows that they are symmetric.
\end{proof}

\begin{remark}
    Stanley \cite{StanleyPersonal} noted (via personal communication) an alternative way of proving Theorem \ref{thm:qsym-sym} and Corollary \ref{cor:factors} starting from the fact that $K$ is a unique factorization domain, which we state here.  
    
    Supopse $h=f\cdot g$ where $f$ and $g$ are in $K$, and assume for contradiction that $f$ (or $g$, but without loss of generality we may consider $f$) is not symmetric.  Then $$h=\pi(h) = \pi(f) \cdot \pi(g)$$ for any permutation $\pi\in S_{\mathbb{N}}$.  We claim that since $f$ is not symmetric and has infinitely many variables, we may obtain infinitely many distinct factors of $h$ of the form $\pi(f)$ in this manner.

    Indeed, since $f$ is not symmetric, there is a monomial $cx_1^{\alpha_1}x_2^{\alpha_2}\cdots x_n^{\alpha_n}$ in $f$ whose coefficient $c$ is distinct from that of $dx_{\pi(1)}^{\alpha_1}\cdots x_{\pi(n)}^{\alpha_n}$ for some fixed permutation $\pi$.  If all but finitely many coefficients of the permuted monomials of $x_1^{\alpha_1}x_2^{\alpha_2}\cdots x_n^{\alpha_n}$ are equal to $c$, then infinitely many of the coefficients are equal to $d$; it follows that either $c$ or $d$ is distinct from infinitely many other coefficients of the monomials with exponents $\alpha$. It follows that there are infinitely many permutations that do not fix $f$.
    
    Thus $h$ has infinitely many distinct factors, contradicting unique factorization.
\end{remark}

\begin{remark}
    Theorem \ref{thm:qsym-sym} can also be seen as a consequence of Corollary 6.5.33 in \cite{GrinbergReiner}, which states that $\QSym$ is a polynomial algebra over $\Sym$.
\end{remark}

\section{Nonsymmetry of Certain Families of Graphs}\label{sec:dag}

The main goal of this section is to prove Theorem \ref{thm:dag}, but we start with some general results on conditions for nonsymmetry of chromatic quasisymmetric functions.

\subsection{General results for Directed Acyclic Graphs}\label{sec:nonsym}
Every labeling of a graph $G$ induces a directed acyclic orientation on $G$ by directing every edge from a smaller label to a larger one. Moreover, any directed acyclic graph can be given a labeling that respects the orientation of the edges. Thus, it equivalent to define the chromatic quasisymmetric function on a directed acyclic graph $G$  where given a proper coloring $\kappa$ of $G$, a directed edge $(u,v)$ is said to have an ascent if $\kappa(u) < \kappa(v)$. Given a directed acyclic graph $G$, let $G^{\rev}$ denote the acyclic graph obtained from $G$ by reversing the orientation of all of the edges.

\begin{lem}\label{lem:rev}
Let $G$ be a directed acyclic graph. Then $X_G(x;q)$ is symmetric if and only if $X_{G^{\rev}}(x;q)$ is symmetric.
\end{lem}

\begin{proof}
As the underlying undirected graph does not change, any proper coloring of $G$ is also a proper coloring of $G^{\rev}$. Any proper coloring $\kappa$ of $G$ with ascent $\asc(\kappa)$ will have ascent $|E| - \asc(\kappa)$ in $G^{\rev}$ where $\lvert E \rvert$ is the number of edges in $G$. Thus, $X_{G^{\rev}}(x;q) = q^{|E|}X_G(x;q^{-1})$ which implies the desired result.
\end{proof}

Any vertex of a directed acyclic graph $G$ that only has edges exiting the vertex is called a \textbf{source}. Likewise, any vertex with only incoming edges is said to be a \textbf{sink}. The following observation is due to Liu \cite{Kevin}.

\begin{lem}\label{lem:same}
    If $G$ has a different number of sources and sinks, then $X_G(x;q)$ is not symmetric.
\end{lem}

\begin{proof}
    Let $a$ be the number of sources and $b$ be the number of sinks in $G$. As $G^{rev}$ has $b$ sources and $a$ sinks, it suffices to assume that $a<b$ by Lemma~\ref{lem:rev}. We will show that in the monomial quasisymmetric function expansion of $X_G(x;q)$, the coefficients of $M_{(a,1^{n-a-b},b)}$ and $M_{(b,1^{n-a-b},a)}$ are different. First, consider a coloring of $G$ such that each source has color $1$, each sink has color $n-a-b+2$, and the intermediate vertices are colored in such a way to achieve the maximum ascent statistic. Since every intermediate vertex has a color between $1$ and $n-a-b+2$, this coloring $\kappa$ has $\mathrm{asc}(\kappa) = |E|$, so $q^{|E|}$ appears with a nonzero coefficient on $M_{(a,1^{n-a-b},b)}$ in the expansion of $X_G(x;q)$. 
    
    Next, consider a coloring of the vertices of $G$ with content $(b,1^{n-a-b},a)$. As $a<b$, there are more sinks in $G$ than sources, and so the coloring must contain a non-source vertex colored with a $1$. Then, an edge directed towards this vertex does not count towards the ascent number of $\kappa$. Thus, no such coloring contains $|E|$ ascents, and so $q^{|E|}$ does not appear in the coefficient of $M_{(b,1^{n-a-b},a)}$ in the expansion of $X_G(x;q)$. Hence $X_G(x;q)$ is not symmetric.
\end{proof}

For the remainder of the paper, we assume that all directed graphs have the same number of sinks and sources.

\begin{definition}
   An \textbf{antichain} in a directed acyclic graph is a set of vertices $\{v_1,\ldots,v_k\}$ such that there is not a directed path from any $v_i$ to $v_j$.
\end{definition}

Note that this matches the definition of an antichain on a poset when we consider the directed acyclic graph as a poset such that $v_i < v_j$ if and only if there is a directed path from $v_i$ to $v_j$.

\begin{lem}\label{lem:antichain}
    Let $G$ be a directed acyclic graph. If $G$ has an antichain whose size is larger than the number of sinks (and the number of sources), then $X_G(x;q)$ is not symmetric.
\end{lem}

\begin{proof}
Let $V$ be the set of vertices of $G$ and let $a$ be the number of sources. By assumption, there exists an antichain $S$ in $G$ with cardinality $|S|>a$. Consider the subset $T$ of $V$ consisting of all vertices that are not in $S$ and have a directed path to a vertex in $S$. (Viewing $G$ as a poset, $T$ consists all elements that are strictly less than at least one element of $S$.) 

We construct a coloring with maximum ascent statistic $|E|$ and content $(1^{|T|},|S|,1^{n-|S|-|T|})$. Color all of the vertices in $S$ with the color $|T|+1$ and the vertices in $T$ with the colors $1$ through $|T|$ such that every edge between vertices in $|T|$ contributes an ascent. Likewise, color the vertices in $V-(S\cup T)$ with the remaining colors such that every edge between these vertices contributes an ascent. As $S$ is an antichain, this construction gives a proper coloring of $G$. Observe that any directed edge crossing between the sets $S$, $T$, and $V-(S \cup T)$ must go from $T$ to $S$, $T$ to $V-(S \cup T)$, or $S$ to $V-(S \cup T)$. The constructed coloring respects this relationship, and therefore has maximum ascent statistic $|E|$. Thus, $q^{|E|}$ appears in the coefficient of $M_{(1^{|T|},|S|,1^{n-|S|-|T|})}$

Now consider a coloring of $G$ with content $(|S|,1^{n-|S|})$. Since $|S| > a$, there exists a nonsource vertex $v$ colored $1$. All incoming edges to this vertex will not contribute an ascent implying that there is no proper coloring of $G$ with content $(|S|,1^{n-|S|})$ and maximum ascent statistic $|E|$. Hence $q^{|E|}$ does not appear in the coefficient of $M_{(|S|,1^{n-|S|})}$, and thus, the chromatic quasisymmetric function is not symmetric.
\end{proof}

\subsection{Directed acyclic graphs with at least two sources}

Consider a connected, directed acyclic graph $G$ with at least two sources. We show that the chromatic quasisymmetric function of $G$ is not symmetric. As before, we will view $G$ as a poset where a vertex $w$ is less than $v$ if and only if there is a directed path from $w$ to $v$. Let $n$ be the number of vertices in $G$ and $a \geq 2$ be its number of sources (and therefore its number of sinks by Lemma ~\ref{lem:same}). Let $S(G)$ be the set of all vertices $v$ of $G$ having at least two sources that are smaller than $v$. 

\begin{lem}
    The set $S(G)$ is nonempty.
\end{lem}
\begin{proof}
    Assume that every sink has exactly one source that is smaller than it. As $G$ has an equal number of sources and sinks, every sink is paired with a unique source. Since $G$ has at least two sinks, this would imply that $G$ is not connected. Thus, at least one sink of $G$ has at least two sources smaller than it, and $S(G)$ is nonempty.
\end{proof}

Given $v \in S(G)$, we define $\stat(v)$ to be the number of nonsource vertices smaller than $v$. Let $k=k_G = 1+\min_{v\in S(G)} \stat(v)$. Denote by $K_{\alpha}^{|E|}$ the set of proper colorings of $G$ having weight $\alpha$ and $|E|$ ascents. We will construct a strictly injective map from $K_{(1^{k},a,1^{n-k-a})}^{|E|}$ to $K_{(a,1^{n-a})}^{|E|}$. Before defining this map, we recall Dilworth's Theorem on posets.

\begin{thm} (Dilworth's Theorem, \cite{Dilworth}) \label{thm:dilworth}
    Let $P$ be a finite poset. Then the largest antichain of $P$ is equal to the minimum number of disjoint chains needed to cover the vertices of $P$.
\end{thm}

By Lemma ~\ref{lem:antichain}, we may assume that the largest antichain of $G$ has size $a$. Thus, $G$ can be covered by $a$ disjoint chains, each of which contains exactly one source and sink. Fix such a minimal chain decomposition $R$ of the graph $G$. We are now ready to define the map of interest.

\begin{definition}
    Let $\varphi_{G,R}\colon K_{(1^{k},a,1^{n-k-a})}^{|E|} \to K_{(a,1^{n-a})}^{|E|}$ where $\varphi_{G,R}(\kappa)$ is the coloring obtained by iterating over every chain in $R$ as follows:
    \begin{itemize}
        \item  If the chain contains a vertex colored $1$, do nothing.
        \item Otherwise, recolor the vertex colored $k+1$ with the color $1$. Then sort the colors in the chain such that they are increasing from the chain's source to the chain's sink.
    \end{itemize}
\end{definition} 

Before showing that $\varphi_{G,R}$ is well-defined, we first prove the following lemma.

\begin{lem} \label{lem:daguod}
    Let $\kappa \in K_{(1^{k},a,1^{n-k-a})}^{|E|}$. Then any vertex in $S(G)$ has color strictly larger than $k+1$ in $\kappa$.
\end{lem}
\begin{proof}
    Given a vertex $v$ in $S(G)$, the number of nonsource vertices smaller than $v$ is at least $k-1$. Moreover, as $v$ is greater than at least two sources, there are at least $k+1$ vertices smaller than $v$. Thus, in order for a proper coloring $\kappa$ of weight $(1^{k},a,1^{n-k-a})$ to have maximum ascent statistic $|E|$, we must have $\kappa(v) \geq k+2$ which proves the lemma.
\end{proof}

\begin{lem}
    The map $\varphi_{G,R}$ is well-defined.
\end{lem}
\begin{proof}
    Let $\kappa \in K_{(1^{k},a,1^{n-k-a})}^{|E|}$. By construction, $\varphi_{G,R}(\kappa)$ has the appropriate weight. It remains to show that the coloring is proper and has maximum ascent statistic $|E|$. Consider two adjacent vertices $v$ and $w$ in $G$ where $v$ is larger than $w$. We show that the coloring of $v$ is strictly larger than that of $w$ in $\varphi_{G,R}(\kappa)$. If $v$ and $w$ lie within the same chain in the minimal chain decomposition $R$, then the color of $v$ is larger than $w$ by construction. Thus, we may assume that $v$ and $w$ lie in different chains in $R$. 

    As the color of every vertex in $\varphi_{G,R}(\kappa)$ is weakly smaller than their respective color in $\kappa$, it suffices to show that $v$ is painted the same color in both $\kappa$ and $\varphi_{G,R}(\kappa)$. Moreover, since $\varphi_{G,R}$ does not recolor any vertex having color larger than $k+1$ in $\kappa$, it suffices to show that the color of $v$ in $\kappa$ is larger than $k+1$.
    
    Since $v$ and $w$ lie in different chains, $v$ must be greater than at least two different sources of $G$: the source in the chain of $v$ and the source in the chain of $w$. This implies that $v \in S(G)$, and $v$ has color larger than $k+1$ in $\kappa$ by Lemma ~\ref{lem:daguod}. Thus, the map $\varphi_{G,R}$ is well-defined.
\end{proof}

\begin{lem} \label{lem:dagsurj}
    The map $\varphi_{G,R}$ is injective, but not surjective.
\end{lem}

\begin{proof}
    A coloring $\kappa \in K_{(1^{k},a,1^{n-k-a})}^{|E|}$ can be recovered from its image $\varphi_{G,R}(\kappa)$ as follows. In $\kappa$, every chain contains a vertex colored $k+1$, and the chains that change are precisely those in $\varphi_{G,R}(\kappa)$ that no longer contain a vertex colored $k+1$.  Moreover, in $\varphi_{G,R}(\kappa)$, every chain in $R$ has its source colored $1$ by the definition of $\varphi_{G,R}$.
    Thus we can recover $\kappa$ by, for each chain of $R$ that does not contain a vertex colored $k+1$, recoloring its source with the color $k+1$ and sorting the colors in the chain such that they are increasing from the chain's source to the chain's sink. Hence, the map $\varphi_{G,R}$ is injective. 
    
    To show that $\varphi_{G,R}$ is not surjective, we find a coloring $\tilde{\kappa} \in K_{(a,1^{n-a})}^{|E|}$ that is not in the image of $\varphi_{G,R}$. Let $v$ be a vertex in $S(G)$ such that $k = \stat(v)+1$. As $\stat(v) = k-1$, there are $k-1$ nonsource vertices smaller than $v$. Construct the coloring $\tilde{\kappa}$ by coloring all sources $1$, coloring the nonsource vertices smaller than $v$ with the colors $2, \ldots, k$ such that every edge has an ascent, coloring $v$ with $k+1$, and coloring the rest of the vertices with the remaining colors so as to obtain the maximum ascent statistic $|E|$. By Lemma ~\ref{lem:daguod}, the vertex $v$ must have color larger than $k+1$ in every coloring in $K_{(1^{k},a,1^{n-k-a})}^{|E|}$. Since $\varphi_{G,R}$ does not change the color of any vertex having color larger than $k+1$, $\tilde{\kappa}$ cannot be in the image of $\varphi_{G,R}$. Thus, the map $\varphi_{G,R}$ is not surjective.
\end{proof}

Lemma ~\ref{lem:dagsurj} gives us the desired result on directed acyclic graphs.

\begin{DagTheorem}
    Let $G$ be a connected, directed acyclic graph. If $G$ has at least two sources, then $X_{G}(x;q)$ is not symmetric.
\end{DagTheorem}

Combining Dilworth's Theorem \cite{Dilworth} with Lemma ~\ref{lem:antichain} and Theorem~\ref{thm:dag}, we also obtain the following.
\begin{cor} \label{cor:dpath}
    Let $G$ be a connected, directed acyclic graph with a symmetric chromatic quasisymmetric function. Then $G$ must have a directed path from its source to sink that includes every vertex of $G$.
\end{cor}

As an additional corollary, we obtain the following characterization of all directed trees with symmetric chromatic quasisymmetric functions thereby settling an open question in \cite{Steph}.

\begin{cor}
    Let $T$ be a directed tree. Then $X_{T}(x;q)$ is symmetric if and only if $T$ is a directed path.
\end{cor}
\begin{proof}
    As directed paths are natural unit interval graphs, the symmetry of the chromatic quasisymmetric functions for directed paths follows from Shareshian and Wachs \cite{SW16}. Let $T$ be a directed tree with a symmetric quasisymmetric function. Theorem ~\ref{thm:dag} implies that $T$ has exactly one source and sink, and Corollary ~\ref{cor:dpath} implies that $T$ must be a directed path from its source to sink. 
\end{proof}

We can also characterize which directed acyclic orientations on cycles have symmetric chromatic quasisymmetric functions. A directed acyclic orientation on a cycle is naturally oriented if the cycle has exactly one source and sink with an edge going from its source to sink.

\begin{cor} \label{cor:cycle}
    Let $C$ be a directed acyclic cycle. Then $X_{C}(q;t)$ is symmetric if and only if $C$ is naturally oriented.
\end{cor}
\begin{proof}
The symmetry of naturally ordered cycles was proven by Ellzey and Wachs ~\cite{EW20}. Let $C$ be a directed acyclic cycle such that $X_{C}(q;t)$ is symmetric. By Corollary ~\ref{cor:dpath}, $C$ has a directed path from its source to its sink that includes every vertex of $C$. In order for $C$ to be a cycle, its source must share edge with its sink implying that $C$ is naturally oriented.
\end{proof}

\begin{remark}
    Ellzey \cite{Ellzey} and Alexandersson and Panova \cite{AP18} studied a generalization of the chromatic quasisymmetric function that allows for directed cycles. In their papers, they show that directed cycles have a symmetric chromatic quasisymmetric function. Thus, with respect to this more general setting, we have that an oriented cycle has a symmetric chromatic quasisymmetric function if and only if it is a directed or naturally oriented cycle.
\end{remark}

\section{Symmetry of generalized mountain graphs}\label{sec:mountains}

We begin with the case of ordinary mountain graphs in which only $k$-cliques are used.

\subsection{Mountain graphs} \label{section:mountain}

We define a mountain graph as follows.

\begin{definition}
    The \textbf{$(p,k)$-mountain graph $M_{p,k}$} is the graph formed by replacing all but one edge in a cycle of length $p+1$ with a $k$-clique (we require $p\geq 2$ and $k\geq 2$). The edge that was not replaced is called the \textbf{bottom edge}.  We call each of the $k$-cliques the \textbf{mountains} of $M_{p,k}$. 
\end{definition}

Notice that $|V(M_{p,k})| = p(k-1)+1$. In this section, we will consider $G = M_{p,k}$. 

We assign an ordering to the vertices by drawing the mountain graph in the plane with the bottom edge connecting the leftmost vertex to the rightmost vertex, and the mountains connected left to right in between, with the vertices in each clique that are not along the $p+1$-cycle all lying strictly between the two endpoints of the clique in left to right order (see Figure \ref{fig:54-mountain}). We then name the vertices $v_1,\ldots, v_{n} = v_{|V(G)|}$ left to right in this order. A vertex along the induced $p+1$-cycle is called a \textbf{lower vertex}, and any other vertex is called an \textbf{upper vertex}. 

\begin{figure}[ht] 
    \centering
\includegraphics{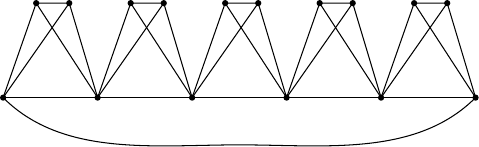}
    \caption{\label{fig:54-mountain}The $(5,4)$-mountain graph}
\end{figure}

Notably, these graphs are not natural unit interval graphs unless $p=2$ and $k=2$ (or a relabeling of a natural unit interval graph when $p=2$ and $k\geq3$).  We will show that if $G$ is a $(p,k)$-mountain graph, then $X_G(x;q)$ is symmetric. In the case of $(p,2)$-mountain graphs, our work recovers the result that the chromatic quasisymmetric functions of naturally labeled cycles are symmetric \cite{EW20}.  Given a proper coloring $\kappa:V(G)\to \N$, define the \textbf{$(a,a+1)$-colored subgraph} to be the induced subgraph on the vertex set $\kappa^{-1}(\{a,a+1\})$.

\begin{lem}\label{lem:mountainsubgraphs}
    Let $G$ be a $(p,k)$-mountain graph, and let $\kappa:V(G)\to \N$ be a proper coloring. Then for all $a$, the $(a,a+1)$-colored subgraph is:
    \begin{itemize} 
    \item A cycle of length $p+1$, or;
    \item A (possibly empty) disjoint union of paths, each of the form $v_{i_1}-v_{i_2}-\cdots-v_{i_j}$, with 
    \begin{itemize}
        \item[$\bullet$] $i_1<i_2<\cdots<i_j$ or
        \item[$\bullet$] $i_1<i_2<\cdots < i_k = n$ and $1 = i_{k+1} < i_{k+2} < \cdots <i_j$ for some $2\leq k\leq j-1$.
    \end{itemize}
    \end{itemize}
\end{lem}

\begin{proof}
    If the $(a,a+1)$-colored subgraph contains all of the lower vertices (which is only possible when $p$ is odd), then no other vertex in any $k$-clique can be colored $a$ or $a+1$. Thus the subgraph is a cycle of length $p+1$. 
    
    Define $G'$ to be the graph $G$ without the bottom edge, so $G'$ is a natural unit interval graph. If vertices $v_1$ and $v_n$ are not both colored at least one of $a$ or $a+1$, then the $(a,a+1)$-colored subgraph is contained in $G'$. By \cite{SW16} (Lemma 4.4), the $(a,a+1)$-colored subgraph is a disjoint union of paths, each of the form $v_{i_1}-v_{i_2}-\cdots-v_{i_j}$, with $i_1<i_2<\cdots<i_j$.

    Now suppose that $v_1$ and $v_n$ are colored $a$ or $a+1$, and not every lower vertex is colored $a$ or $a+1$. Then in $G'$, vertices $v_1$ and $v_n$ are in different connected components of the $(a,a+1)$-colored subgraph. Again, since $G'$ is a natural unit interval graph, these connected components form paths $v_1 = v_{i_{k+1}} - \cdots - v_{i_{j}}$ and $v_{i_1} - \cdots - v_{i_{k}} = v_n$ such that $i_{k+1} < \cdots <i_j$ and $i_1<\cdots < i_k$. In $G$, this yields a path of the form $v_{i_1}-v_{i_2}-\cdots-v_{i_j}$, which includes the bottom edge. 
\end{proof}

Notice that for any coloring of a $(p,k)$-mountain graph, if the $(a,a+1)$-colored subgraph contains a path, then the middle vertices of the path cannot be upper vertices of the mountain graph.

Let $G$ be a $(p,k)$-mountain graph, and define $K_{a,a+1}(G)$ to be the set of proper colorings of $G$ such that the $(a,a+1)$-colored subgraph does not include the bottom edge. 

\begin{prop}\label{prop:K}
    There is an ascent-preserving involution on $K_{a,a+1}(G)$ which swaps the number of occurrences of the colors $a$ and $a+1$.
\end{prop}

\begin{proof}
    As above, define $G'$ to be the graph $G$ without the bottom edge. Since $G'$ is a natural unit interval graph, the above statement follows from the proof of \cite{SW16}, Theorem 4.5. We summarize this argument here for clarity. 

    From Lemma \ref{lem:mountainsubgraphs}, for $\kappa \in K_{a,a+1}(G)$, each connected component of the $(a,a+1)$-colored subgraph is a path of the form $v_{i_1}-v_{i_2}-\cdots - v_{i_j}$, with $i_1<i_2<\cdots < i_j$. Let $\psi_a(\kappa)$ be the coloring of $G$ obtained by swapping $a$ and $a+1$ exactly when the number of vertices in the connected path is odd (if this number is even, leave the coloring of the component unchanged). Notice that $\psi_a(\kappa)$ is still a proper coloring. 

    Since each path with an even number of vertices contains the same number of vertices colored $a$ and $a+1$, it follows that $\psi_a(\kappa)$ swaps the number of instances of colors $a$ and $a+1$. The paths with an odd number of vertices have an even number of edges, so swapping $a$ and $a+1$ preserves the number of ascents in these paths. Further, since $a$ and $a+1$ are adjacent colors, ascents between a vertex in one of these paths and a vertex colored $b$ with $b\neq a$ and $b\neq a+1$ are preserved. 
\end{proof}

Now define $L_{a,a+1}(G)$ to be the set of proper colorings of $G$ such that the $(a,a+1)$-colored subgraph includes the bottom edge.

\begin{prop}\label{prop:lake}
     There is an ascent-preserving automorphism on $L_{a,a+1}(G)$ which swaps the number of occurrences of the colors $a$ and $a+1$. 
\end{prop}

To prove this, we first establish several lemmas about operations that we will put together to show the symmetry.

\begin{definition}
    Let $a>1$, and let $\kappa\in L_{a,a+1}(G)$ be a coloring with maximal color $c$. We define a new coloring $\cycle(\kappa) \in L_{a-1,a}(G)$ as follows: 
    \begin{itemize}
        \item Change all vertices colored $1$ to color $c+1$.
        \item For each vertex colored $c+1$:
        \begin{itemize}
            \item If the vertex is a lower vertex, do nothing.
            \item If the vertex is an upper vertex, suppose that it is the $i$-th upper vertex from the left in its clique. Then reorder the upper vertices in the clique so that this vertex is the $i$-th upper vertex from the right, while the relative order of the other vertices is preserved. 
        \end{itemize}
        \item Reduce the value of all colors by $1$.
    \end{itemize} 
\end{definition}

Note that the cycle operation applies a cyclic permutation to the content of $\kappa$.  An example of the construction $\cycle(\kappa)$ is given in Figure \ref{fig:34-mountain-cycle}. 

\begin{figure}
    \centering
\includegraphics{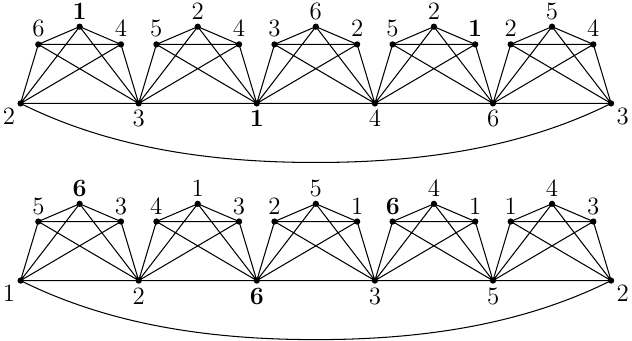}

    \caption{\label{fig:34-mountain-cycle}Above, an element $\kappa \in L_{2,3}(G)$, and below, the output $\cycle(\kappa) \in L_{1,2}(G)$.}
\end{figure}

\begin{lem}\label{lem:cycle}
    Let $a>1$. Then $\cycle \colon L_{a,a+1}(G) \to L_{a-1,a}(G)$ is an ascent-preserving bijection.
\end{lem}

\begin{proof}
    As each step of $\cycle$ is invertible, the map $\cycle$ is bijective. It remains to show that it is ascent preserving. Suppose that $\kappa\in L_{a,a+1}(G)$ has largest color $c$. We will show that the number of ascents within each clique is preserved, and that any ascent formed by the bottom edge is preserved. 

    First, suppose that $C$ is a clique in $\kappa$ with no vertex colored $1$. Then the only change to $C$ is that the color of every vertex is reduced by $1$, which preserves the relative orders of the vertices incident to each edge. Thus $\cycle(\kappa)$ has the same number of ascents within this clique.

    Next, suppose that $C$ is a clique of $\kappa$ containing at least one (and thus, exactly one) vertex colored $1$. We have two cases. If $v$ is a lower vertex, then $v$ cannot be the leftmost vertex $v_1$ or rightmost vertex $v_n$, since $\kappa\in L_{a,a+1}$ implies $\kappa(v_1)$ and $\kappa(v_n)$ are $a$ or $a+1$. Thus $v$ is connected to two of the cliques of $G$. Since $v$ is colored $1$, it forms an ascent with each of the $k-1$ vertices in the clique to its right, and no ascents with the vertices in the clique to its left. When the color of $v$ is changed to $c+1$, since this color is greater than any present in $\kappa$, $v$ forms an ascent with each vertex in the clique to its left, and none with the clique on its right. Since $\kappa$ is a proper coloring, no other vertex in the cliques adjacent to $v$ could have color $1$, so no other vertices gain the color $c+1$. Reducing the value of each color by $1$ again preserves the relative order of the vertices on each edge, so $\cycle(\kappa)$ has the same number of ascents within the cliques adjacent to $v$. 

    If $v$ is an upper vertex, then since $v$ has color $1$, it forms an ascent with the vertices to its right within its clique - suppose there are $\ell$ such ascents. Recoloring $v$ with $c+1$, and then moving $v$ so that it is to the left of $\ell$ of the vertices in the clique preserves the number of ascents formed on an edge incident to $v$. Since the relative orders of the other vertices in the clique are preserved, the ascents between these vertices are preserved. Reducing the color of each vertex by $1$ again preserves the number of ascents, so $\cycle(\kappa)$ has the same number of ascents within the clique containing $v$. 

    Lastly, for the bottom edge, the two vertices incident to the bottom edge are not labeled $1$ because $a>1$ and $\kappa\in L_{a,a+1}(G)$. Reducing the value of these colors by $1$ does not change whether the bottom edge forms an ascent. Combining these cases, we get that $\asc(\kappa) = \asc(\cycle(\kappa))$, as desired. 
\end{proof}

\begin{definition} \label{def:reflect}
    Suppose $a=1$, and let $\kappa\in L_{a,a+1}(G)=L_{1,2}(G)$ with largest color $c$.  Define $\reflect(\kappa)$ to be the coloring obtained from $\kappa$ as follows:
    \begin{itemize}
        \item Reflect $\kappa$ horizontally. (i.e. color vertex $i$ with the color $\kappa(n+1-i)$.)
        \item Recolor all $1$'s as $2$'s and vice versa.
        \item Recolor all $3$'s as $c$'s, all $4$'s as $c-1$'s, and so on.
        \item Within each $k$-clique $C$ of $\kappa$, consider the original positions of any $1$ or $2$ appearing in the upper entries (which can be no positions, a single index $i$, or two indices $i$ and $j$ from left to right).  Then move any new $1$'s or $2$'s in the image of $C$ under reflection to positions $i$ or $i,j$ in left to right order, keeping the relative order of the other upper entries the same.        
    \end{itemize}
\end{definition}

See Figure ~\ref{fig:reflect-12} for an example of this construction.

\begin{figure}
\centering
\includegraphics{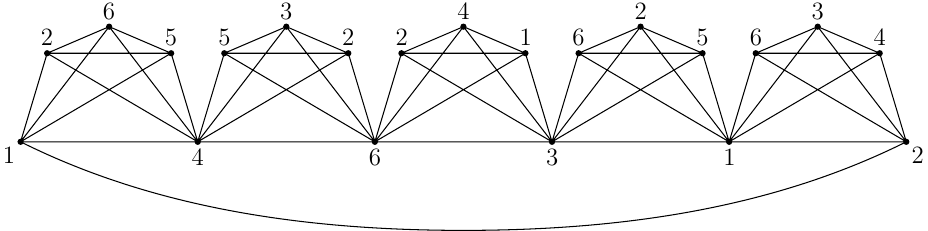}
    \caption{
    \label{fig:reflect-12} The output $\reflect(\cycle(\kappa))$ for $\kappa$ as in Figure ~\ref{fig:34-mountain-cycle}.}
\end{figure}

\begin{lem} \label{lem:reflect}
    The map $\reflect:L_{1,2}(G)\to L_{1,2}(G)$ is an ascent-preserving automorphism. 
\end{lem}

\begin{proof}
As each step of $\reflect$ is invertible, it is bijective. It remains to show that it is ascent preserving. As in Definition ~\ref{def:reflect}, let $\kappa \in L_{1,2}(G)$ and let $c$ denote the largest color in $\kappa$. Consider an edge $e$ of $G$ whose vertices are colored by $\kappa$ with colors between $3$ and $c$ inclusive. In $\reflect(\kappa)$, these vertices are sent to a new edge $\tilde{e}$ of $G$. The horizontal reflection of the colors in $\kappa$ followed by the swapping of all colors between $3$ and $c$ preserves the relative order of the colors of vertices incident to the edge $e$. Thus, $e$ contributes to $\asc(G)$ if and only if $\tilde{e}$ contributes to $\asc(\reflect(\kappa))$. 

Similarly, any edge of $G$ whose vertices are colored with only the numbers $1$ and $2$ in $\kappa$ contributes to $\asc(\kappa)$ if and only if its image in $\reflect(\kappa)$ contributes to $\asc(\reflect(\kappa))$. Note that this implies the bottom edge under $\kappa$ contributes an ascent if and only if it contributes an ascent in $\reflect(\kappa)$.

We now examine the number of ascents in every $k$-clique under both colorings. By the previous paragraphs it suffices to only consider edges where exactly one vertex is colored $1$ or $2$. Let $C$ be a $k$-clique of $G$ under the coloring $\kappa$ and let $\tilde{C}$ be its image under the coloring $\reflect(\kappa)$. We break into cases based on the colors of the lower vertices of $C$.

If the lower vertices of $C$ are colored $1$ and $2$ in any order, then the edges where exactly one of the vertices is colored $1$ or $2$ are precisely the edges between the upper and lower vertices of $C$ and $\tilde{C}$. As $1$ and $2$ are the smallest colors, the number of ascents for edges between upper vertices and lower vertices in both $C$ and $\tilde{C}$ is $k-2$.

If the lower vertices of $C$ are colored between $3$ and $c$ inclusive, then the positions of any vertices colored $1$ or $2$ in $C$ are the same in $\tilde{C}$. This implies that the number of ascents for edges where exactly one of the vertices is colored $1$ or $2$ is the same in both $C$ and $\tilde{C}$.  

Now assume exactly one of the lower vertices of $C$ is colored $1$ or $2$. We define a $k$-clique to be ``left-colored" if the only lower vertex colored $1$ or $2$ in the $k$-clique is its left lower vertex. Similarly define the notion of being ``right-colored". There exists a natural pairing of left-colored $k$-cliques to right-colored $k$-cliques by pairing a left-colored $k$-clique with the leftmost right-colored $k$-clique that sits to its right in $G$. Note that as $v_1$ and $v_n$ are colored $1$ and $2$ (in some order), such a pairing must exist. Moreover, any $k$-clique falling between a paired left-colored and right-colored $k$-clique cannot have any lower vertices colored $1$ or $2$. We show that the total number of ascents within a pair of such $k$-cliques is equal to the total number of ascents within its image in $\reflect(\kappa)$.

Without loss of generality, assume that $C$ is left-colored and let $D$ be its corresponding right-colored $k$-clique. Additionally, let $\tilde{C}$ and $\tilde{D}$ be the images of $C$ and $D$ under the coloring $\reflect(\kappa)$ respectively. Recall that the position of any upper vertex colored $1$ or $2$ within a $k$-clique is the same as its position within the image of the $k$-clique in $\reflect(\kappa)$. Thus, the number of ascents for edges between upper vertices where exactly one vertex is colored $1$ or $2$ is the same for $C$ and $\tilde{C}$ and for $D$ and $\tilde{D}$. Moreover, $C$ and $\tilde{D}$ will each have an ascent between their lower vertices while $\tilde{C}$ and $D$ do not.

We now consider edges between upper and lower vertices where exactly one vertex is colored $1$ or $2$. If $C$ has no upper vertices colored $1$ or $2$, then there are $k-2$ ascents between the left lower vertex and upper vertices of $C$. If however, $C$ has an upper vertex $w$ colored $1$ or $2$, then there are $k-3$ ascents between the left lower vertex and upper vertices not colored $1$ or $2$. Note that the edge between $w$ and the right lower vertex of $C$ also forms an ascent. In either case, there are $k-2$ ascents for edges between upper and lower vertices of $C$ where exactly one vertex is colored $1$ or $2$. Similarly, $\tilde{D}$ has $k-2$ ascents for edges between upper and lower vertices where exactly one vertex is colored $1$ or $2$. A routine check tells us that there are no such ascents for the cliques $D$ and $\tilde{C}$. Therefore, the total number of ascents in $C$ and $D$ is the same as the total number of ascents in $\tilde{C}$ and $\tilde{D}$.

Combining all of the above cases gives us the desired result $\asc(\reflect(\kappa))=\asc(\kappa)$.
    
\end{proof}

We now have the tools to prove Proposition ~\ref{prop:lake}.

\begin{proof}
    Let $\kappa\in L_{a,a+1}(G)$, and suppose its color content is $$(c_1,\ldots,c_n).$$  
    We first apply $\cycle$ exactly $a-1$ times, so that $\kappa':=\cycle^{a-1}(\kappa)$ is in $L_{1,2}(G)$, with content $$(c_a,c_{a+1},c_{a+2},\ldots,c_n,c_1,\ldots,c_{a-1}),$$ 
    and has the same number of ascents as $\kappa$.  Then we apply $\reflect$ to form $\kappa''\in L_{1,2}(G)$ with content equal to $$(c_{a+1},c_a,c_{a-1},\ldots,c_1,c_n,\ldots,c_{a+2}).$$  Finally, fix a reduced word for the permutation that reverses entries 3 through $n$, say, $$(\sigma_3)(\sigma_4\sigma_3)(\sigma_5\sigma_4\sigma_3)\cdots (\sigma_{n-1}\cdots \sigma_3)$$
    We then apply the automorphisms from Proposition ~\ref{prop:K} to apply these transpositions $\sigma_i$ to the content in order, to obtain $\kappa'''$ with content 
    $$(c_{a+1},c_a,c_{a+2},\ldots,c_n,c_1,\ldots,c_{a-1}).$$
    Finally, we apply $\cycle^{-1}$ to $\kappa'''$ exactly $a-1$ times to obtain a coloring $\kappa''''$ with content 
    $$(c_1,\ldots,c_{a-1},c_{a+1},c_a,c_{a+2},\ldots,c_n).$$  The automorphism $\kappa\to \kappa''''$ is well-defined, bijective, and preserves ascents, because every step of the construction has these properties.  This completes the proof.
\end{proof}

Combining Propositions ~\ref{prop:K} and ~\ref{prop:lake}, we obtain the following result.

\begin{thm}
    Let $M_{p,k}$ be a $(p,k)$-mountain graph. Then $X_{M_{p,k}}(x;q)$ is symmetric.
\end{thm}

\subsection{Bottomless Mountain graphs} \label{sec:bottomless}

From the family of $(p,k)$-mountain graphs, we construct another family of graphs whose chromatic quasisymmetric function is symmetric.

\begin{definition}
    For $k\geq 3$, the \textbf{$(p,k)$-bottomless mountain graph} $B_{p,k}$ is the $(p,k)$-mountain graph $M_{p,k}$ where all edges between lower vertices are removed except for the bottom edge.
\end{definition}

\begin{figure}[ht] \label{fig:54-botmount}
    \centering
\includegraphics{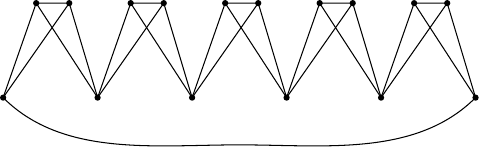}
    \caption{The $(5,4)$-bottomless mountain graph}
\end{figure}

We keep the same ordering on the vertices of $B_{p,k}$ as for $M_{p,k}$ as well as the definitions of the bottom edge, lower vertices, and upper vertices. To show the chromatic quasisymmetric function of $B_{p,k}$ is symmetric, it suffices to show that the ascent-preserving automorphisms presented in Section ~\ref{section:mountain} for the colorings of $M_{p,k}$ hold true for $B_{p,k}$ as well.

Let $L_{a,a+1}(B_{p,k})$ (resp. $K_{a,a+1}(B_{p,k})$) be the set of proper colorings such that the $(a,a+1)$-colored subgraph does (resp. does not) include the bottom edge. As removing the bottom edge from any bottomless mountain results in a natural unit interval order, the involution given in Proposition~\ref{prop:K} is also an ascent-preserving involution on $K_{a,a+1}(B_{p,k})$ that swaps the number of occurrences of the colors $a$ and $a+1$.

Lemma ~\ref{lem:cycle} clearly holds for the graphs $B_{p,k}$, that is the $\cycle$ map gives an ascent-preserving map from $L_{a,a+1}(B_{p,k})$ to $L_{a-1,a}(B_{p,k})$ for all $a > 1$. In order for the proof of Proposition~\ref{prop:lake} to pass through for the graphs $B_{p,k}$, it remains to show that the $\reflect$ map is ascent-preserving automorphism on $ L_{1,2}(B_{p,k})$.

\begin{lem} \label{lem:bottomlessreflect} Let $\kappa \in L_{1,2}(B_{p,k})$. Then $\asc(\reflect(\kappa))=\asc(\kappa)$.    
\end{lem}

\begin{proof}
As in Lemma ~\ref{lem:reflect} it suffices to consider edges where exactly one vertex is colored either $1$ or $2$ and the other vertex is not. To show that $\asc(\reflect(\kappa))=\asc(\kappa)$, we examine the number of ascents coming from these edges in every bottomless $k$-clique. Let $C$ be a bottomless $k$-clique of $G$ under the coloring $\kappa$ and let $\tilde{C}$ be its image under the coloring $\reflect(\kappa)$. We break into cases based on the colors of the lower vertices of $C$.

All cases follow directly from the proof of Lemma ~\ref{lem:reflect} with the exception of when the lower vertices are colored $1$ or $2$. In this case, it is possible for both lower vertices of $C$ to have the same color and for an upper vertex to be colored with the opposite color in $\{1,2\}$. The number of ascents in both $C$ and $\tilde{C}$ coming from edges where exactly one vertex is colored $1$ or $2$ is then $k-3$. Thus, $\asc(\reflect(\kappa))=\asc(\kappa)$ for all $\kappa \in L_{1,2}(B_{p,k})$.
\end{proof}

As the analogues of Lemmas ~\ref{lem:cycle} and ~\ref{lem:reflect} hold for the bottomless graphs $B_{p,k}$, the proof of Proposition ~\ref{prop:lake} gives an ascent-preserving automorphism on $L_{a,a+1}(B_{p,k})$ that swaps the number of occurrences of the colors $a$ and $a+1$. Pairing this ascent-preserving automorphism on $L_{a,a+1}(B_{p,k})$ with the ascent-preserving involution on $K_{a,a+1}(B_{p,k})$ gives an ascent-preserving automorphism on the set of all proper colorings of $B_{p,k}$ that interchanges the number of occurrences of the colors $a$ and $a+1$. This proves the desired symmetry.

\begin{thm}
    Let $B_{p,k}$ be a $(p,k)$-bottomless-mountain graph. Then $X_{B_{p,k}}(x;q)$ is symmetric. 
\end{thm}

\subsection{Mixed Mountain Graphs}

A $(p,k)$-mixed mountain graph is a $(p,k)$-mountain graph where some number of the $k$-clique mountains are replaced with bottomless $k+1$-clique mountains.

\begin{figure}
\begin{center}
    \includegraphics[scale = .9]{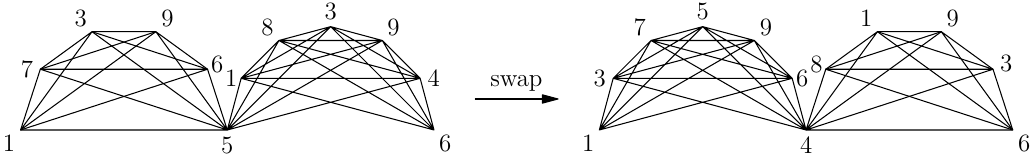}
\end{center}
\caption{\label{fig:swap} An example of the swap map, showing that we can swap a $k$-clique with an adjacent bottomless $k+1$-clique without changing the chromatic quasisymmetric function.}
\end{figure}

We wish to define the $\swap$ map, which will swap a $k$-clique with a bottomless $k+1$-clique to its right in a way that preserves ascents.  We first construct several auxiliary definitions.

Let $G$ be a mixed mountain graph.  Suppose $\mathcal{U}$ is a $k$-clique in $G$, and there is a bottomless $k+1$-clique $\mathcal{W}$ immediately to its right, with shared vertex $v\in \mathcal{U}\cap \mathcal{W}$.  Then we name the remaining vertices of $\mathcal{U}$ by $u,u_1,\ldots,u_{k-2}$ from left to right, and the remaining vertices of $\mathcal{W}$ by $w_1,w_2,\ldots,w_{k-1},w$. 

\begin{definition}
 Define the graph $\swap_v(G)$ to be the graph formed by replacing $\mathcal{U}$ with a bottomless $k+1$-clique, and replacing $\mathcal{W}$ with a $k$-clique.  
\end{definition}

\begin{definition}
    Suppose $\kappa$ is a proper coloring of $G$.  Then we define the ``smaller-colored'' and ``larger-colored'' sets of upper vertices by $$S_{\mathcal{U}}=\{u_i\in \mathcal{U}\mid \kappa(u_i)<\kappa(v)\}\hspace{2cm} L_{\mathcal{U}}=\{u_i\in \mathcal{U}\mid \kappa(u_i)>\kappa(v) \}.$$  Define $S_{\mathcal{W}}$ and $L_{\mathcal{W}}$ similarly.
\end{definition}

\begin{example}
  In the graph in Figure \ref{fig:swap} at left, we have $S_{\mathcal{U}}=\{u_2\}$, $L_{\mathcal{U}}=\{u_1,u_3,u_4\}$, $S_{\mathcal{W}}=\{w_1,w_3,w_5\}$, and $L_{\mathcal{W}}=\{w_2,w_4\}$.  These correspond, respectively, to the color $3$ on the left, the colors $7,9,6$ on the left, the colors $1,3,4$ on the right, and the colors $8,9$ on the right.
\end{example}

\begin{definition}\label{def:special_vertex}
    The \textbf{special vertex} of a triple $(\mathcal{U},\mathcal{W},\kappa)$ is one of the $w_i$ vertices, defined as follows.
    \begin{itemize}
        \item If $|S_{\mathcal{U}}|<|S_{\mathcal{W}}|$, then list the vertices of $S_{\mathcal{U}}\cup S_{\mathcal{W}}$ in decreasing order of vertex color, breaking ties by listing the elements of $S_{\mathcal{U}}$ first.  Associate an open parenthesis to elements of $S_{\mathcal{U}}$, closed parenthesis to elements of $S_{\mathcal{W}}$, and pair them. Let $w_i$ be the leftmost unpaired entry from $S_{\mathcal{W}}$ in this list.
       
        \item If $|L_{\mathcal{U}}| < |L_{\mathcal{W}}|$, then order the vertices of $L_{\mathcal{U}}\cup L_{\mathcal{W}}$ in increasing order of vertex color, breaking ties by listing the elements of $L_{\mathcal{U}}$ first. Associate an open parenthesis to elements of $L_{\mathcal{U}}$, a closed parenthesis to elements of $L_{\mathcal{W}}$, and pair them. Let $w_i$ be the leftmost unpaired entry from $L_{\mathcal{W}}$ in this list. 
    \end{itemize}
\end{definition}

\begin{example}
    In the example above, we have $|S_{\mathcal{U}}|<|S_{\mathcal{W}}|$, so we list the vertices in decreasing order of vertex color.  We list the colors, then the vertex labels, and then the parenthesization in the table below:
    \begin{center}
    \begin{tabular}{cccc}
           4 & 3 & 3 & 1 \\
           $w_5$ & $u_2$ & $w_3$ & $w_1$ \\
           ) & ( & ) & )
    \end{tabular}
    \end{center}
    The leftmost unpaired entry is $w_5$, with color $4$, so this is the special vertex.
\end{example}

\begin{lem}
    The special vertex $w_i$ for the triple $(\mathcal{U},\mathcal{W},\kappa)$ exists and is well-defined. 
\end{lem}

\begin{proof}
    Since $\kappa$ is a proper coloring, the sets of colors $\kappa(\mathcal{U}):=\{\kappa(u_1),\ldots, \kappa(u_{k-2})\}$ and $\kappa(\mathcal{W}):=\{\kappa(w_1),\ldots, \kappa(w_{k-1})\}$ contain $k-2$ and $k-1$ elements, respectively. Further, $\mathcal{U} = S_{\mathcal{U}}\cup L_{\mathcal{U}}$ and $\mathcal{W} = S_{\mathcal{W}}\cup L_{\mathcal{W}}$ are unions of disjoint sets. Since $|\mathcal{U}|+1= |\mathcal{W}|$, we must have exactly one of $|S_{\mathcal{U}}|<|S_{\mathcal{W}}|$ or $|L_{\mathcal{U}}| < |L_{\mathcal{W}}|$. In either case, when associating parenthesis to the list of vertices, there are more vertices from $\mathcal{W}$ in the list, so there is some associated unpaired closed parenthesis. Hence the special vertex is well-defined.  
\end{proof}

We now define the $\swap$ map for a coloring $\kappa$ of $G$. 

\begin{definition}
    We define $\swap(\kappa):V(\swap_v(G))\to \mathbb{N}$ as follows. Color the vertices outside of $\mathcal{U}\cup \mathcal{W}$, as well as vertices $u$ and $w$, with the same color as $\kappa$ for $G$. Color the single vertex $v\in \mathcal{U}\cap \mathcal{W}$ with $\kappa(w_i)$, where $w_i$ was the special vertex of $(\mathcal{U},\mathcal{W},\kappa)$. Then color the upper vertices of the new $k$-clique using the colors $\kappa(w_1),\ldots,\kappa(w_{i-1}),\kappa(w_{i+1}),\ldots, \kappa(w_{k-1})$ in the same relative order as $\kappa(u_1),\ldots,\kappa(u_{k-2})$, and the upper vertices of the new bottomless $k+1$-clique using the colors $\kappa(u_1),\ldots, \kappa(u_{k-2}),\kappa(v)$ in the same relative order as $\kappa(w_1),\ldots,\kappa(w_{k-1})$.  
\end{definition}

See Figure \ref{fig:swap} for an example.

\begin{lem} \label{lem:swap_ascents}
    Suppose that $G$ is a $(p,k)$-mixed mountain graph and $\kappa$ is a proper coloring of $G$. Then $\swap(\kappa)$ is a proper coloring, and $\kappa$ and $\swap(\kappa)$ have the same number of ascents.  
\end{lem}

\begin{proof}
    Let $\kappa$ be a proper coloring of $G$, and let $\swap(\kappa)$ be the associated coloring of $\swap_v(G)$. 
    
    First, we show that $\swap(\kappa)$ is a proper coloring. Let $x$ and $x'$ be adjacent vertices in $\swap_v(G)$. If neither of $x$ and $x'$ are in either of the swapped cliques $\mathcal{U}$ and $\mathcal{W}$, then they are assigned the same colors in $\kappa$ and $\swap(\kappa)$, so $\swap(\kappa)(x) \neq \swap(\kappa)(x')$. If $x$ and $x'$ are both upper vertices in the new $k$-clique or bottomless $k+1$-clique, then $\swap(\kappa)$ assigns them distinct colors from $\kappa(\mathcal{W}\setminus\{w_i\})$ or $\kappa(\mathcal{U}\cup \{v\})$, respectively.  If $x$ is an upper vertex and $x'$ is not, then $x'$ is one of the neighboring lower vertices of the new $k$-clique or bottomless $k+1$-clique. In either case, the colors of the vertices neighboring $x$ in $\swap_v(G)$ are a subset of the colors of the vertices neighboring the vertex with color $\swap(\kappa)(x)$ in $G$. If $x$ and $x'$ are both lower vertices, then they must be the lower vertices of the new $k$-clique. In this case, $\swap(\kappa)$ assigns $x$ and $x'$ colors from two adjacent vertices in $\mathcal{W}$. 
    Hence $x$ and $x'$ are assigned different colors in $\swap(\kappa)$. 

    Next, we show that $\kappa$ and $\swap(\kappa)$ have the same number of ascents. First, if $e$ is an edge between any two vertices not in $\mathcal{U}$ or $\mathcal{W}$, then $\swap(\kappa)$ assigns the same color to these vertices as $\kappa$, so it has an ascent exactly if $\kappa$ has an ascent. If $e$ is an edge between the upper vertices of the new $k$-clique, then these vertices were colored so as to maintain the relative order of colors of upper vertices of $\mathcal{U}$, so the edge $e$ has an ascent if and only if the corresponding edge in the original $k$-clique $\mathcal{U}$ has an ascent. Likewise, if $e$ is an edge between the upper vertices of the new bottomless $k+1$-clique, then since the ordering of the colors of the upper vertices is maintained, $e$ has an ascent if and only if the corresponding edge in $\mathcal{W}$ has an ascent in $\kappa$. 

    Consider the set of edges between the left vertex $u$ and upper vertices of the new bottomless $k+1$-clique. Since these upper vertices were assigned colors from $\kappa(\mathcal{U})$, which were the set of colors adjacent to $u$ in $G$, these edges contain the same number of ascents as in the coloring $\kappa$. By the same argument, the edges between the upper vertices of the new $k$-clique and $w$ contain the same number of ascents in $\kappa$ and $\swap(\kappa)$.  

    Finally, consider the set of edges between the middle vertex $v$ and an upper vertex in either clique.   In $\kappa$, the number of ascents formed by these edges is $|S_{\mathcal{U}}|+|L_{\mathcal{W}}|$. If $|S_{\mathcal{U}}|<|S_{\mathcal{W}}|$, then consider the pairing defined in Definition~\ref{def:special_vertex}, and let $w_i$ be the special vertex with color $c=\kappa(w_i)$. Let $a=\kappa(v)$ be the color of $v$.  Then $c=\swap(\kappa)(v)$, and since $c$ was an unpaired color, each set of paired vertices are either both colored less or both colored greater than $c$. In particular, if $u_j$ is paired with $w_k$, then the edges $u_j$ --- $v$ --- $w_k$ contain exactly one ascent in both $\kappa$ and $\swap(\kappa)$. By the definition of $w_i$, all unpaired vertices from $S_{\mathcal{U}}$ and $S_{\mathcal{W}}$ have color at most $c$, so the number of ascents in $\kappa$ between these unpaired vertices and $v$ is the same as the number of ascents in $\swap(\kappa)$ between these vertices.  (Note that the edge between $v$ and $w_i$ corresponds to an edge between the new vertex labeled $a$ in the $k$-clique and $v$, which will still be an ascent in $\swap(\kappa)$ if and only if $v$ --- $w_i$ formed an ascent in $\kappa$.)
    
    Since entries from $L_{\mathcal{U}}$ and $L_{\mathcal{W}}$ are given color greater than $\kappa(v)$ and $\swap(\kappa)(v)$, we get that the number of ascents formed by edges between $\mathcal{U}$, $\mathcal{W}$, and $v$ in $\swap(\kappa)$ is also $|S_{\mathcal{U}}|+|L_{\mathcal{W}}|$. 
    
    An analogous argument shows that ascents are preserved in the case $|L_{\mathcal{U}}|<|L_{\mathcal{W}}|$.  Therefore we have $\asc(\kappa)=\asc(\swap(\kappa))$. 
\end{proof}

\begin{lem}\label{lem:swap_bijection}
    The map $\kappa\mapsto \swap(\kappa)$ is a bijection between colorings of $G$ and $\swap_v(G)$. 
\end{lem}

\begin{proof}
    We first claim that it suffices to show that the map is injective.  If it is, we can compose swap maps to form an injective map from colorings of the mixed mountain graph $G_0$ having exactly $N$ $k$-cliques followed by $M$ bottomless $k+1$-cliques, to the graph $G_1$ having $M$ bottomless $k+1$-cliques followed by $N$ $k$-cliques, which by symmetry have the same total number of proper colorings.  Thus this composed map is a bijection, so every intermediate swap map is a bijection as well (and every mixed mountain graph can be obtained via some sequence of swaps starting from $G_0$).
    
    Suppose for contradiction that there exist colorings $\kappa$ and $\rho$ of $G$ which both map to the coloring $\swap(\kappa)=\swap(\rho)$ on $\swap_v(G)$. For simplicity, we set $\kappa':= \swap(\kappa)$ and $\rho' = \swap(\rho)$. For a complete subgraph $H\subset G$, let $\kappa(H)$ be the set of colors on vertices of $H$. By construction of $\kappa'$, since $\kappa'=\rho'$, we must have $\kappa(\mathcal{U}) = \rho(\mathcal{U})$ and $\kappa(\mathcal{W}\setminus\{v\}) = \rho(\mathcal{W}\setminus \{v\})$. If $\kappa(v) = \rho(v)$, then since $\swap$ preserves the relative orders of the upper vertices, we must have that $\kappa=\rho$. 
    
    Now assume that $\kappa(v) \neq \rho(v)$, and without loss of generality suppose that $\kappa(v)<\rho(v)$. We will show that $\kappa'(v) \neq \rho'(v)$. Let $S_{\mathcal{U}}^{\kappa}$ be the set of vertices of $\mathcal{U}$ colored smaller than $\kappa(v)$ and $S_{\mathcal{U}}^{\rho}$ be the set of vertices of $\mathcal{U}$ colored smaller than $\rho(v)$. Similarly define $L_{\mathcal{U}}^{\kappa}, L_{\mathcal{U}}^{\rho}$, $S_{\mathcal{W}}^{\kappa}$, $S_{\mathcal{W}}^{\rho}$, $L_{\mathcal{W}}^{\kappa}$, and $L_{\mathcal{W}}^{\rho}$. We have several cases:

    Suppose that $|S_{\mathcal{U}}^{\kappa}| < |S_{\mathcal{W}}^{\kappa}|$ and $|L_{\mathcal{U}}^{\rho}|<|L_{\mathcal{W}}^{\rho}|$. Then in $\rho'$, the color for $v$ is chosen from the colors of $L_{\mathcal{W}}^{\rho}$, and in $\kappa'$, the color for $v$ is chosen from $S_{\mathcal{U}}^{\kappa}$. So we have $\rho'(v) > \rho(v) > \kappa(v) > \kappa'(v)$. 

    Suppose that $|S_{\mathcal{U}}^{\kappa}|<|S_{\mathcal{W}}^{\kappa}|$ and $|S_{\mathcal{U}}^{\rho}| < |S_{\mathcal{W}}^{\rho}|$.  Let $s^{\kappa}$ be the sequence consisting of vertices of $S_{\mathcal{U}}^{\kappa}\cup S_{\mathcal{W}}^{\kappa}$, ordered in decreasing order of vertex color, breaking ties by listing elements of $S_{\mathcal{U}}^{\kappa}$ first. Let $s^{\rho}$ be the similarly defined word on $S_{\mathcal{U}}^{\rho}\cup S_{\mathcal{W}}^{\rho}$.  
    
    Since $\kappa(\mathcal{W}\setminus\{v\}) = \rho(\mathcal{W}\setminus\{v\})$ and the relative order for the colors of upper vertices in $\mathcal{W}$ must match the relative order of colors in $\kappa'$ for the upper vertices of the $k+1$-bottomless clique , we have that $\kappa$ and $\rho$ color the upper vertices of $\mathcal{W}$ exactly the same. So, since $\kappa(v)<\rho(v)$, we have $S_{\mathcal{W}}^{\kappa}\subseteq S_{\mathcal{W}}^{\rho}$. Next, because $\kappa(\mathcal{U}) = \rho(\mathcal{U})$, and the order of upper vertices of $\mathcal{U}$ is the same in $\kappa$ and $\rho$, we have that $S_{\mathcal{U}}^{\kappa} \subseteq S_{\mathcal{U}}^{\rho}$. Further, $\kappa$ and $\rho$ only differ in the positions of upper vertices $u_j$ with $\kappa(v)\leq\kappa(u_j)\leq\rho(v)$. Therefore $s^{\kappa}$ is a subword of $s^{\rho}$, and in fact since they are both decreasing we have $s^\rho=\tilde{s}s^{\kappa}$ for some word $\tilde{s}$.
    
    Additionally, the vertex $x\in \mathcal{U}$ with color $\rho(x) = \kappa(v)$ is the first character to the left of $s^{\kappa}$ in $s^{\rho}$, since it is colored smallest among $S_{\mathcal{U}}^{\rho}\setminus S_{\mathcal{U}}^{\kappa}$ and there can not be any colors tied with $\kappa(v)$. Since $x$ is associated to an open parenthesis in $s^{\rho}$, it pairs with the leftmost unpaired closed parenthesis in $s^{\kappa}$. Therefore, the leftmost unpaired closed parenthesis in $s^{\kappa}$ is different than in $s^{\rho}$, so $\kappa'(v)\neq \rho'(v)$, contradicting our assumption. 
    
    The two cases when $|L_{\mathcal{U}}^{\kappa}|<|L_{\mathcal{W}}^{\kappa}|$ follow similarly. Therefore, the map sending $(G,\kappa)$ to $(\swap_v(G),\swap(\kappa))$ is injective. 
\end{proof}

Since $\swap$ is a bijection, we can equivalently say that $\swap^{-1}$ is an ascent-preserving map on colorings for graphs where the swapped $k$-clique is to the right of the bottomless $k+1$-clique. Because of Lemmas ~\ref{lem:swap_ascents} and ~\ref{lem:swap_bijection}, we conclude the following. 

\begin{thm} \label{thm:equalmixed}
    If $G$ is a $(p,k)$-mixed mountain graph, and $\swap_v(G)$ is the $(p,k)$-mixed mountain graph obtained by swapping an adjacent $k$-clique and bottomless $k+1$-clique, then $X_{G}(x;q) = X_{\swap_v(G)}(x;q)$. 
\end{thm}

\begin{remark}
    Note that the proof of the above theorem does not rely on the presence of the bottom edge in the mixed-mountain graph. Thus, this also implies a result on natural unit interval graphs. Let $G$ be a natural unit interval graph obtained by deleting the bottom edge from a $(p,k)$-mixed mountain graph, and let $H$ be the natural unit interval graph obtained by deleting the bottom edge from $\swap_v(G)$. Then $X_{G}(x;q) = X_{H}(x;q)$.
\end{remark}

As a consequence of Theorem ~\ref{thm:equalmixed}, it suffices to show the chromatic quasisymmetric functions of $(p,k)$-mixed mountain graphs where every $k$-clique is to the left of a bottomless $(k+1)$-clique are symmetric. We denote by $M_{p,k,m}$ the $(p,k)$-mixed mountain graph with $m$ $k$-cliques followed by $p-m$ bottomless $(k+1)$-cliques. Let $L_{a,a+1}(M_{p,k,m})$ denote the set of colorings whose $(a,a+1)$-colored subgraph includes the bottom edge, and $K_{a,a+1}(M_{p,k,m})$ denote the set of proper colorings of $M_{p,k,m}$ such that the $(a,a+1)$-colored subgraph does not include the bottom edge. As before, we demonstrate ascent-preserving automorphisms on the sets $L_{a,a+1}(M_{p,k,m})$ and $K_{a,a+1}(M_{p,k,m})$ which swap the number of occurrences of the colors $a$ and $a+1$.

Removing the bottom edge of $M_{p,k,m}$ results in a unit interval graph. Thus, the proof of Proposition ~\ref{prop:K} immediately gives an ascent-preserving involution on $K_{a,a+1}(M_{p,k,m})$. In order to find an ascent-preserving automorphism on $L_{a,a+1}(M_{p,k,m})$, we make use of the $\swap$, $\cycle$ and $\reflect$ maps. We first show that $\cycle$ and $\reflect$ are still ascent-preserving when generalized to the set $L_{a,a+1}(M_{p,k,m})$.

 \begin{lem} \label{lem:cyclemixed}
    Suppose $\kappa \in L_{a,a+1}(M_{p,k,m})$ with $a>1$. Then $\asc(\cycle(\kappa))=\asc(\kappa)$.
\end{lem}
\begin{proof}
    By the proof of Lemma ~\ref{lem:cycle} and the discussion in Section~\ref{sec:bottomless}, it remains to consider the case where the lower vertex $v$ shared between the adjacent $k$-clique and bottomless $(k+1)$-clique is colored $1$. This vertex will have exactly $k-1$ ascents in $\kappa$ coming from edges to upper vertices of the bottomless $(k+1)$-clique. In $\cycle(\kappa)$, this vertex will also have exactly $k-1$ ascents coming from all its edges with vertices in the $k$-clique. Thus, $\asc(\cycle(\kappa))=\asc(\kappa)$.
\end{proof}

Unlike for mountain and bottomless mountain graphs, the $\reflect$ map on mixed mountain graphs is not an automorphism on the set $L_{1,2}(M_{p,k,m})$ but rather a map from $L_{1,2}(M_{p,k,m})$ to $L_{1,2}(M^{\rev}_{p,k,m})$. Nevertheless, this map still bijective and preserves ascents.

\begin{lem} \label{lem:reflectmixed}
Let $\kappa \in L_{1,2}(M_{p,k,m})$. Then $\asc(\reflect(\kappa))=\asc(\kappa)$.    
\end{lem}

\begin{proof}
    Let $c$ denote the largest color in $\kappa$. By the proofs of Lemma ~\ref{lem:reflect} and ~\ref{lem:bottomlessreflect}, it remains to consider the case when a ``left-colored" $k$-clique $C$ is paired with a ``right-colored" bottomless $(k+1)$-clique $D$. Let $\tilde{C}$ and $\tilde{D}$ be their respective images in the coloring $\reflect(\kappa)$. Given that the $\reflect$ map preserves ascents on edges whose vertices are colored $1$ and $2$ or whose vertices are colored with the numbers from the palette $\{3, \ldots, c\}$, it suffices to consider edges where exactly one vertex is colored either $1$ or $2$. In addition, the relative position of any upper vertex colored $1$ or $2$ within $C$ or $D$ will remain the same in $\tilde{C}$ and $\tilde{D}$ respectively. This implies that the number of ascents for edges between upper vertices where exactly one vertex is colored $1$ or $2$ is the same for $C$ and $\tilde{C}$ and the same for $D$ and $\tilde{D}$.

    We now examine edges in $C$, $D$, $\tilde{C}$, and $\tilde{D}$ where exactly one vertex is colored $1$ or $2$ and where at least one vertex is a lower vertex. As $C$ is left-colored, its left lower vertex is colored $1$ or $2$ and there is an ascent along its edge to the right lower vertex. Moreover, by the same argument as in Lemma ~\ref{lem:reflect}, there are $k-2$ ascents for edges between upper and lower vertices in $C$ where exactly one vertex is colored $1$ or $2$. This implies that $C$ has a total of $k-1$ ascents for edges where exactly one vertex is colored $1$ or $2$ and where at least one vertex is a lower vertex. The bottomless $(k+1)$-clique $D$ and the $k$-clique $\tilde{C}$ have no such ascents as they are both right-colored. If an upper vertex of $\tilde{D}$ is colored $1$ or $2$, there are $k-2$ ascents from its left lower vertex to upper vertices and one ascent from the upper vertex colored $1$ or $2$ to the right lower vertex. If there are no upper vertices of $\tilde{D}$ colored $1$ or $2$, there are $k-1$ ascents from its left lower vertex to the upper vertices. In all cases $\tilde{D}$ has $k-1$ ascents for edges between upper and lower vertices where exactly one vertex is colored $1$ or $2$. As $\tilde{D}$ does not have an edge between its lower vertices, this is also the number of ascents of $\tilde{D}$ for edges where exactly one vertex is colored $1$ or $2$ and where at least one vertex is a lower vertex. Therefore, $C$ and $D$ have a total of $k-1$ ascents for such edges, and so do their images $\tilde{C}$ and $\tilde{D}$. Thus, $\asc(\reflect(\kappa))=\asc(\kappa)$.
\end{proof}

\begin{prop} \label{prop:Lmixed}
    There is an ascent-preserving automorphism on $L_{a,a+1}(M_{p,k,m})$ which swaps the number of occurrences of the colors $a$ and $a+1$.
\end{prop}
\begin{proof}
    Let $\kappa\in L_{a,a+1}(M_{p,k,m})$, and suppose its color content is $$(c_1,\ldots,c_n).$$  We start by applying $\cycle$ to $\kappa$ exactly $a-1$ times, so that $\kappa':=\cycle^{a-1}(\kappa)$ is in $L_{1,2}(M_{p,k,m})$, with content $$(c_a,c_{a+1},c_{a+2},\ldots,c_n,c_1,\ldots,c_{a-1}),$$ and has the same number of ascents as $\kappa$. By Lemmas ~\ref{lem:swap_ascents} and ~\ref{lem:swap_bijection}, we now can apply $\swap$ $a(p-a)$ times to $\kappa'$ to obtain $\kappa''$ in $L_{1,2}(M^{\rev}_{p,k,m})$ with the same content and number of ascents as $\kappa'$. Then we apply $\reflect^{-1}$ to form $\kappa'''\in L_{1,2}(M_{p,k,m})$ with content equal to $$(c_{a+1},c_a,c_{a-1},\ldots,c_1,c_n,\ldots,c_{a+2}).$$  Finally, fix a reduced word for the permutation that reverses entries 3 through $n$, say, $$(\sigma_3)(\sigma_4\sigma_3)(\sigma_5\sigma_4\sigma_3)\cdots (\sigma_{n-1}\cdots \sigma_3).$$
    We then apply the automorphisms from Proposition ~\ref{prop:K} to apply these transpositions $\sigma_i$ to the content in order, to obtain $\kappa^{(4)}$ with content 
    $$(c_{a+1},c_a,c_{a+2},\ldots,c_n,c_1,\ldots,c_{a-1}).$$
    Finally, we apply $\cycle^{-1}$ to $\kappa^{(4)}$ exactly $a-1$ times to obtain a coloring $\kappa^{(5)}$ in $L_{a,a+1}(M_{p,k,m})$ with content 
    $$(c_1,\ldots,c_{a-1},c_{a+1},c_a,c_{a+2},\ldots,c_n).$$  The automorphism $\kappa\to \kappa^{(5)}$ is well-defined, bijective, and preserves ascents, because every step of the construction has these properties. This completes the proof.
\end{proof}

Propositions ~\ref{prop:K} and ~\ref{prop:Lmixed} and Theorem ~\ref{thm:equalmixed} imply the desired symmetry of the chromatic quasisymmetric function of $(p,k)$-mixed mountain graphs.

\begin{MountainTheorem}
    Let $G$ be a $(p,k)$-mixed mountain graph. Then $X_{G}(x;q)$ is symmetric.
\end{MountainTheorem}

\printbibliography

\end{document}